\documentclass[a4paper,12pt]{article}

\usepackage{amsmath,amssymb,amsthm, latexsym}
\usepackage{graphicx,graphics}

\theoremstyle{plain}
\newtheorem{Thm}{Theorem}[section]
\newtheorem{Lem}[Thm]{Lemma}
\newtheorem{Prop}[Thm]{Proposition}

\theoremstyle{definition}
\newtheorem{Ex}[Thm]{Example}
\newtheorem*{Rem}{Remark}
\newcommand{\ex}{\mathrm{ex}} 
\newcommand{\Max}{\mathbf{Max}} 

\title{On the rank functions of $\mathcal{H}$-matroids}

\author{
{\sc Yoshio SANO}
\footnote{Division of Information Engineering, 
Faculty of Engineering, Information and Systems, 
University of Tsukuba, 
Ibaraki 305-8573, Japan. 
\textit{E-mail:} \texttt{sano@cs.tsukuba.ac.jp}}
}

\date{}

\begin{document}

\maketitle

\begin{abstract}
The notion of $\mathcal{H}$-matroids was introduced 
by U. Faigle and S. Fujishige 
in 2009 as a general model for matroids and the greedy algorithm. 
They gave a characterization of $\mathcal{H}$-matroids 
by the greedy algorithm. 
In this note, we give a characterization of some 
$\mathcal{H}$-matroids by rank functions. 
\end{abstract}



%%%%%%%%%%%%%%%%%%%%%%%%%%%%%%%%%%%%%%%%%%%%%%%%%%%%%%%%%%%%%%%%%%%%%%
\section{Introduction and Main Result} 
%%%%%%%%%%%%%%%%%%%%%%%%%%%%%%%%%%%%%%%%%%%%%%%%%%%%%%%%%%%%%%%%%%%%%%

The notion of matroids was introduced by H. Whitney \cite{Whitney} in 1935  
as an abstraction of the notion of linear independence in a vector space. 
Many researchers have studied and extended the theory of matroids 
(cf. \cite{F-book, KLS, Oxley, Schrijver, Welsh}). 
In 2009, 
U. Faigle and S. Fujishige \cite{FF} 
introduced the notion of $\mathcal{H}$-matroids 
as a general model for matroids and the greedy algorithm. 
They gave a characterization of $\mathcal{H}$-matroids 
by the greedy algorithm. 
In this note, we give a characterization of 
the rank functions of $\mathcal{H}$-matroids 
that are \emph{simplicial complexes},  
for {\it any} family $\mathcal{H}$. 
Our main result is as follows. 

\begin{Thm}\label{th:main}
Let $E$ be a finite set and 
let $\rho:2^E \to \mathbb{Z}_{\geq 0}$ be a set function on $E$. 
Let $\mathcal{H}$ be a family of subsets of $E$ 
with $\emptyset, E \in \mathcal{H}$. 
Then, $\rho$ is the rank function of an $\mathcal{H}$-matroid 
$(E,\mathcal{I})$ 
if and only if $\rho$ is a 
normalized unit-increasing function 
satisfying the $\mathcal{H}$-extension property. 
\begin{itemize}
\item[{\rm (E)}] {\rm ($\mathcal{H}$-extension property)} \\
For $X \in 2^E$ and $H \in \mathcal{H}$ with $X \subseteq H$, 
if $\rho(X)=|X|<\rho(H)$, \\
then there exists $e \in H \setminus X$ such that 
$\rho(X \cup \{e\})=\rho(X)+1$. 
\end{itemize}
Moreover, if $\rho$ is a 
normalized unit-increasing set function on $E$ 
satisfying the $\mathcal{H}$-extension property 
and $\mathcal{I}:=\{X \in 2^E \mid \rho(X)=|X| \}$, 
then $(E,\mathcal{I})$ is an $\mathcal{H}$-matroid with rank function $\rho$ 
and $\mathcal{I}$ is a simplicial complex. 
\end{Thm}

This note is organized as follows. 
Section 2 gives some definitions and preliminaries 
on $\mathcal{H}$-matroids. 
In Section 3, we give a proof of Theorem \ref{th:main} 
and 
an example which shows 
$\mathcal{H}$-matroids that are not simplicial complexes 
are not 
characterized 
only 
by their rank functions. 

%%%%%%%%%%%%%%%%%%%%%%%%%%%%%%%%%%%%%%%%%%%%%%%%%%%%%%%%%%%%%%%%%%%%%%%%%%%
\section{Preliminaries}
%%%%%%%%%%%%%%%%%%%%%%%%%%%%%%%%%%%%%%%%%%%%%%%%%%%%%%%%%%%%%%%%%%%%%%%%%%

Let $E$ be a nonempty finite set 
and let $2^E$ denote the family of all subsets of $E$. 
For any family $\mathcal{I}$ of subsets of $E$, 
the \textit{extreme-point operator} $\ex_{\mathcal{I}}:\mathcal{I} \to 2^E$ 
and the \textit{co-extreme-point operator} $\ex^*_{\mathcal{I}}:\mathcal{I} \to 2^E$ 
associated with $\mathcal{I}$ are defined 
as follows: 
\begin{eqnarray*}
\ex_{\mathcal{I}}(I) &:=& 
\{ e \in I \mid I \setminus \{e\} \in \mathcal{I} \} \quad 
(I \in \mathcal{I}), \\
\ex^*_{\mathcal{I}}(I) &:=& 
\{ e\in E \setminus I \mid I \cup \{e\} \in \mathcal{I} \} \quad 
(I \in \mathcal{I}). 
\end{eqnarray*}
For any family $\mathcal{I} \subseteq 2^E$, 
we denote the set of maximal elements of $\mathcal{I}$ 
with respect to set inclusion 
by $\Max(\mathcal{I})$. 

Let $\mathcal{I}$ be a nonempty family of subsets of a finite set $E$. 
The family $\mathcal{I}$ is called \emph{constructible} if it satisfies 
\begin{itemize}
\item[{\rm (C)}] 
$\ex_{\mathcal{I}}(I) \neq \emptyset$ 
\quad for all $I \in \mathcal{I} \setminus \{\emptyset\}$. 
\end{itemize}

\noindent
Note that {\rm (C)} implies $\emptyset \in \mathcal{I}$. 
We call $I \in \mathcal{I}$ a \emph{base} of $\mathcal{I}$ 
if $\ex^*_{\mathcal{I}}(I) =\emptyset$. 
We denote by $\mathcal{B}(\mathcal{I})$ the family of bases of $\mathcal{I}$, i.e., 
$\mathcal{B}(\mathcal{I}):=\{I \in \mathcal{I} \mid \ex^*_{\mathcal{I}}(I) = \emptyset \}$. 
By definition, it holds that $\mathcal{B}(\mathcal{I}) \supseteq \Max(\mathcal{I})$. 

A constructible family $\mathcal{I}$ induces 
a \emph{(base) rank function} $\rho:2^E \to \mathbb{Z}_{\geq 0}$ via
\[
\rho (X) = \textrm{max}_{ B \in \mathcal{B}(\mathcal{I})} |X \cap B| 
= \textrm{max}_{ I \in \mathcal{I}} |X \cap I| 
= \textrm{max}_{I \in \Max(\mathcal{I})} |X \cap I|. 
\]

The following is easily verified by definitions. 

\begin{Lem}\label{lem:UI}
The rank function $\rho$ of a constructible family 
is normalized (i.e. $\rho(\emptyset) = 0$) 
and satisfies the unit-increase property 
\begin{itemize}
\item[{\rm (UI)}] 
$\rho(X) \leq \rho(Y) \leq \rho(X) +|Y \setminus X|$ 
\qquad for all $X \subseteq Y \subseteq E$.
\end{itemize}
\end{Lem}

\noindent
Remark that, by putting $X = \emptyset$, we obtain  
\begin{itemize}
\item[{\rm (UI)$'$}] 
$0 \leq \rho(Y) \leq |Y|$ 
\qquad for all $Y \subseteq E$.
\end{itemize}

The \emph{restriction} of $\mathcal{I}$ to a subset $A \in 2^E$ is the family 
$\mathcal{I}^{(A)} := \{I \in \mathcal{I} \mid I \subseteq A \}$. 
Note that every restriction of a constructible family is constructible.

A \emph{simplicial complex}
is a family $\mathcal{I} \subseteq 2^E$ 
such that 
$X \subseteq I \in \mathcal{I}$ 
implies $X \in \mathcal{I}$. 
We can easily check the following lemmas on simplicial complexes.  

\begin{Lem}\label{lem:SCchara}
A family $\mathcal{I} \subseteq 2^E$ is a simplicial complex 
if and only if 
$\ex_{\mathcal{I}}(I) = I$ holds for any $I \in \mathcal{I}$. 
\end{Lem}

\begin{proof}
The lemma follows from the definitions 
of a simplicial complex and $\ex_{\mathcal{I}}(\cdot)$.
\end{proof}

\begin{Lem}\label{lem:sc-abc}
Let $\mathcal{I} \subseteq 2^E$ be a simplicial complex 
and let $X \in 2^E$. Then, 
\begin{itemize}
\item[{\rm (a)}]
$\mathcal{B}(\mathcal{I}) = \Max(\mathcal{I})$. 
\item[{\rm (b)}]
For $X \in 2^E$, 
$X \in \mathcal{I}$ if and only if $\rho(X)=|X|$. 
\item[{\rm (c)}]
For $H \in 2^E$, 
the family 
$\mathcal{I}^{(H)} \subseteq 2^H$ 
is a simplicial complex. 
\end{itemize}
\end{Lem}

\begin{proof}
(a): 
Suppose that there exists an element 
$B \in \mathcal{B}(\mathcal{I}) \setminus \Max(\mathcal{I})$. 
Then, since $B$ is not maximal in $\mathcal{I}$, 
there exists $I \in \mathcal{I}$ such that 
$B \subsetneq I$. 
For any $e \in I \setminus B$, we have 
$B \cup \{e\} \in \mathcal{I}$ since $B \cup \{e\} \subseteq I$ 
and $\mathcal{I}$ is a simplicial complex. 
Therefore $e \in \ex^*_{\mathcal{I}}(B)$. 
But this is a contradiction to 
$B \in \mathcal{B}(\mathcal{I})$. 

(b):
If $X \in \mathcal{I}$, then $\rho(X)=\max_{I \in \mathcal{I}} |X \cap I|=|X|$. 
Take $X \in 2^E$ with $\rho(X)=|X|$. 
Then there exists $I \in \mathcal{I}$ 
such that $|X \cap I|=\rho(X)=|X|$. 
Therefore, $X \subseteq I$. 
Since $\mathcal{I}$ is a simplicial complex, 
we have $X \in \mathcal{I}$. 

(c): 
Take any $X \in 2^H$ and 
$I \in \mathcal{I}^{(H)} := \{I \in \mathcal{I} \mid I \subseteq H \}$ 
with $X \subseteq I$. 
Since $\mathcal{I}$ is a simplicial complex, $X \in \mathcal{I}$. 
Since $X \subseteq H$, we have $X \in \mathcal{I}^{(H)}$. 
\end{proof}

We now recall the definitions of 
an \emph{$\mathcal{H}$-independence system}
and an \emph{$\mathcal{H}$-matroid}, 
which were introduced by Faigle and Fujishige \cite{FF}. 
Let $E$ be a finite set and 
let $\mathcal{H}$ be a family of subsets of $E$ 
with $\emptyset, E \in \mathcal{H}$. 
A constructible family $\mathcal{I} \subseteq 2^E$ is called 
an \emph{$\mathcal{H}$-independence system} if
\begin{itemize}
\item[{\rm (I)}] 
for all $H \in \mathcal{H}$, 
there exists $I \in \mathcal{I}^{(H)}$ 
such that $|I|=\rho(H)$. 
\end{itemize}

\noindent
An \emph{$\mathcal{H}$-matroid} is a pair $(E,\mathcal{I})$ 
of the set $E$ 
and an $\mathcal{H}$-independence system $\mathcal{I}$ 
satisfying the following property: 
\begin{itemize}
\item[(M)] 
for all $H \in \mathcal{H}$, 
all the bases $B$ of 
$\mathcal{I}^{(H)}$ 
have the same cardinality $|B|= \rho(H)$. 
\end{itemize}

%%%%%%%%%%%%%%%%%%%%%%%%%%%%%%%%%%%%%%%%%%%%%
\section{Proof of Theorem 1.1}
%%%%%%%%%%%%%%%%%%%%%%%%%%%%%%%%%%%%%%%%%%%%%%%

First, we see an example which shows that 
$\mathcal{H}$-matroids that are not simplicial complexes 
are not characterized by their rank functions. 

\begin{Ex}
Let $E = \{1,2,3\}$ and $\mathcal{H}=\{\emptyset, E\}$. 
Let 
\begin{eqnarray*}
\mathcal{I}_1 &=& \{\emptyset, \{2\}, \{1,2\}, \{2,3\} \}, \\
\mathcal{I}_2 &=& \{\emptyset, \{1\}, \{3\}, \{1,2\}, \{2,3\} \}, \\ 
\mathcal{I}_3 &=& \{\emptyset, \{1\}, \{2\}, \{3\}, \{1,2\}, \{2,3\} \}. 
\end{eqnarray*}
Then $(E,\mathcal{I}_1)$, $(E,\mathcal{I}_2)$, and $(E,\mathcal{I}_3)$ are 
$\mathcal{H}$-matroids with the same rank function $\rho:2^E \to \mathbb{Z}_{\geq 0}$ 
such that $\rho(\emptyset)=0$, 
$\rho(\{1\}) = \rho(\{2\}) = \rho(\{3\}) = \rho(\{1,3\})=1$, and 
$\rho(\{1,2\}) = \rho(\{2,3\}) = \rho(\{1,2,3\})=2$. 
\end{Ex}

\noindent
Therefore, we cannot distinguish $\mathcal{H}$-matroids in general 
by their rank functions. 
More generally, the following holds. 

\begin{Prop}\label{prop:same}
For any constructible families $\mathcal{I}$ and $\mathcal{I}'$ 
with $\Max(\mathcal{I})=\Max(\mathcal{I}')$, 
the rank function $\rho'$ associated with $\mathcal{I}'$ 
coincides with the rank function $\rho$ associated with $\mathcal{I}$. 
\end{Prop}

\begin{proof} 
For any $X \in 2^E$, 
it holds that
\[ 
\rho(X)
=\textrm{max}_{I \in \Max(\mathcal{I})} |X \cap I| 
= \textrm{max}_{I \in \Max(\mathcal{I}')} |X \cap I| 
=\rho'(X)
\]
since $\Max(\mathcal{I}) = \Max(\mathcal{I}')$. 
\end{proof}

In the following, we give a proof of Theorem \ref{th:main}. 

\begin{Lem}\label{clm:same}
For any constructible family, 
there exists a simplicial complex 
such that 
their rank functions are the same. 
\end{Lem}

\begin{proof}
Let $\mathcal{I} \subseteq 2^E$ be a constructible family. 
Define $\mathcal{I}':=\{X \in 2^E \mid X \subseteq I 
\text{ for some } I \in \mathcal{I} \}$. 
Then it is clear that $\mathcal{I}'$ is a simplicial complex. 
Obviously each $Y \in \Max(\mathcal{I})$ is maximal 
in $\mathcal{I}'$, and $\mathcal{I}'$ 
does not have new maximal members. 
Therefore $\Max(\mathcal{I}) = \Max(\mathcal{I}')$. 
Note that any simplicial complex is a constructible family. 
By Proposition \ref{prop:same}, 
the rank functions of $\mathcal{I}$ and $\mathcal{I}'$ are the same. 
\end{proof}

\begin{Lem}\label{clm:ExProp}
Let $\rho:2^E \to \mathbb{Z}_{\geq 0}$ be 
the rank function of an $\mathcal{H}$-matroid $(E,\mathcal{I})$, 
where $\mathcal{I}$ is a simplicial complex. 
Then $\rho$ satisfies the $\mathcal{H}$-extension property. 
\end{Lem}

\begin{proof}
Take $X \in 2^E$ and $H \in \mathcal{H}$ with $X \subseteq H$, 
and suppose that $\rho(X)=|X|<\rho(H)$. 
By Lemma \ref{lem:sc-abc} (c), $\mathcal{I}^{(H)}$ is a simplicial complex 
since $\mathcal{I}$ is a simplicial complex. 
Note that $\mathcal{B}(\mathcal{I}^{(H)})=\Max(\mathcal{I}^{(H)})$ 
by Lemma \ref{lem:sc-abc} (a). 
By Lemma \ref{lem:sc-abc} (b), $X \in \mathcal{I}$. 
Therefore $X \in \mathcal{I}^{(H)}$, 
and $X$ is not a base of $\mathcal{I}^{(H)}$ by (I) and (M) 
since $\rho(X) < \rho(H)$. 
Thus there exists $B \in \mathcal{I}$ such that 
$X \subsetneq B \subseteq H$ and $|B|=\rho(H)$. 
Take any element $e \in B \setminus X \subseteq H \setminus X$. 
Then $X \cup \{e\} \in \mathcal{I}$ 
since $X \cup \{e\} \subseteq B \in \mathcal{I}$ 
and 
$\mathcal{I}$ is a simplicial complex. Hence 
it follows 
that $\rho(X \cup \{e\})=|X \cup \{e\}| =|X|+1=\rho(X)+1$. 
\end{proof}

\begin{Lem}\label{clm:core}
Let $\rho:2^E \to \mathbb{Z}_{\geq 0}$ be 
a normalized unit-increasing function 
satisfying the $\mathcal{H}$-extension property 
for some family $\mathcal{H} \subseteq 2^E$ 
with $\emptyset,E \in \mathcal{H}$. 
Put 
\[
\mathcal{I}_{\rho}:=\{X \in 2^E \mid \rho(X)=|X| \}. 
\]
Then $(E,\mathcal{I}_{\rho})$ is an $\mathcal{H}$-matroid and $\mathcal{I}_{\rho}$ 
is a simplicial complex. 
\end{Lem}

\begin{proof}
First we show that $\mathcal{I}_{\rho}$ is a simplicial complex. 
Take any $I \in \mathcal{I}_{\rho} \setminus \{\emptyset\}$ 
and any $e \in I$. 
Then we have $\rho(I)=|I|$. 
Since $\rho$ is unit-increasing, 
we have 
$\rho(I) \leq 
\rho(I \setminus \{e\}) +1$ 
and thus $\rho(I \setminus \{e\}) \geq \rho(I)-1 =|I|-1
=|I \setminus \{e\}|$. 
By (UI) and $\rho(\emptyset)=0$, we also have  
$\rho(I \setminus \{e\}) \leq 
0+ |I \setminus \{e\}|$ 
and thus $\rho(I \setminus \{e\}) \leq |I \setminus \{e\}|$. 
Therefore we have $\rho(I \setminus \{e\})= |I \setminus \{e\}|$ 
and thus $I \setminus \{e\} \in \mathcal{I}_{\rho}$. 
By Lemma \ref{lem:SCchara}, 
$\mathcal{I}_{\rho}$ is a simplicial complex. 
Hence it follows from definitions that 
$\mathcal{I}_{\rho}$ satisfies (C) and (I). 

Now we show that $\mathcal{I}_{\rho}$ satisfies (M). 
Take any $H \in \mathcal{H}$. 
Suppose that there exist $B_1,B_2 \in \mathcal{B}(\mathcal{I}_{\rho}^{(H)})$ 
such that $|B_1| \neq |B_2|$. 
Without loss of generality, we may assume that $|B_1|<|B_2| \leq \rho(H)$. 
Note that $\rho(B_1)=|B_1|$ and $\rho(B_2)=|B_2|$. 
Then, by (E), there exists $e \in H \setminus B_1$ 
such that $\rho(B_1 \cup \{e\})=\rho(B_1)+1=|B_1|+1=|B_1 \cup \{e\}|$. 
Thus we have $B_1 \cup \{e\} \in \mathcal{I}_{\rho}$ 
with $B_1 \cup \{e\} \subseteq H$. 
But this is a contradiction to the assumption that $B_1$ is a base of 
$\mathcal{I}_{\rho}^{(H)}$. 
Thus $\mathcal{I}_{\rho}$ satisfies (M). 
Hence $(E, \mathcal{I}_{\rho})$ is an $\mathcal{H}$-matroid. 
\end{proof}

\begin{proof}[Proof of Theorem \ref{th:main}]
It follows from Lemmas \ref{lem:UI}, 
\ref{clm:same}, 
\ref{clm:ExProp}, 
and \ref{clm:core}. 
\end{proof}

\begin{Rem}
\emph{Strict cg-matroids} which were 
introduced by 
S. Fujishige, G. A. Koshevoy, and Y. Sano \cite{fks} 
in 2007 
can be considered as $\mathcal{H}$-matroids $(E, \mathcal{I})$ 
where $\mathcal{H}$ is an abstract convex geometry 
and $\mathcal{I} \subseteq \mathcal{H}$. 
The rank functions $\rho:\mathcal{H} \to \mathbb{Z}_{\geq 0}$ 
of strict cg-matroids $(E, \mathcal{H}; \mathcal{I})$ 
are characterized in \cite{Sano1}. 
For more study on \emph{cg-matroids}, see \cite{cgmat2}. 
\end{Rem}

\begin{Rem}
Faigle and Fujishige gave a characterization of the 
rank functions 
$\mathcal{H}$-matroids when $\mathcal{H}$ is a closure space 
(see \cite[Theorem 5.1]{FF}).
\end{Rem}

%%%%%%%%%%%%%%%%%%%%%%%%%%%%%%%%%%%%%%%%%%%%%%%%%%%%%%%%%%%%%%%%%%%%%%
\section*{Acknowledgments}
%%%%%%%%%%%%%%%%%%%%%%%%%%%%%%%%%%%%%%%%%%%%%%%%%%%%%%%%%%%%%%%%%%%%%%

The author is grateful to the anonymous referees 
for careful reading and valuable comments. 
This work was supported by JSPS KAKENHI Grant Number 15K20885.

%%%%%%%%%%%%%%%%%%%%%%%%%%%%%%%%%%%%%%%%%%%%%%%%%%%%%%%%%%%%%%%%%%%%%%

\end{document}